\documentclass[12pt, a4paper, reqno]{amsart}

\usepackage[english]{babel}

\usepackage[T1]{fontenc}
\usepackage{lmodern}

\usepackage{amsmath}
\usepackage{amssymb}
\usepackage{amsthm}

\usepackage{url}

\usepackage{hyperref} % also for the table of contets

\usepackage{enumitem}
\usepackage[normalem]{ulem}

\renewcommand{\epsilon}{\varepsilon}
\renewcommand{\theta}[0]{\vartheta}
\renewcommand{\phi}[0]{\varphi}

\newcommand{\Z}{\mathbb{Z}}

\newcommand{\ppar}{\ \par}

\newcommand{\Span}[1]{\left\langle\, #1 \,\right\rangle}

\newcommand{\Set}[1]{\left\{ #1 \right\}}

\newcommand{\extra}[0]{\widehat}

\DeclareMathOperator{\Aut}{Aut}

\DeclareMathOperator{\Frat}{Frat}

\newtheorem{dummy}{Dummy}
\numberwithin{dummy}{section}
\numberwithin{figure}{section}

\newtheorem{theorem}[dummy]{Theorem}

\newtheorem{lemma}[dummy]{Lemma}

\newtheorem{prop}[dummy]{Proposition}

\theoremstyle{definition}
\newtheorem{definition}[dummy]{Definition}

\newtheorem{example}[dummy]{Example}

\theoremstyle{remark}

\newtheorem{remark}[dummy]{Remark}

\newtheorem{question}[dummy]{Question}

\makeatletter
\def\imod#1{\allowbreak\mkern10mu({\operator@font mod}\,\,#1)}
\makeatother

\numberwithin{equation}{section}

\usepackage{amsthm, mathtools, amssymb, amsfonts, stmaryrd, latexsym,
mathrsfs, todonotes, eufrak, color} 

\begin{document}

\date{5 June 2024, 10:28 CEST --- Version 4.03%
}

\title[Elements of prime order]
      {Elements of prime order
      \\
      in the upper central series
      \\
      of a group
      of prime-power order}
      
\author{A. Caranti}

\address[A.~Caranti]%
 {Dipartimento di Matematica\\
  Universit\`a degli Studi di Trento\\
  via Sommarive 14\\
  I-38123 Trento\\
  Italy} 

\email{andrea.caranti@unitn.it} 

\urladdr{https://caranti.maths.unitn.it/}

\author{C. M. Scoppola}

\address[C.~M.~Scoppola]
{Dipartimento di Ingegneria e Scienze dell'In\-for\-ma\-zio\-ne e Matematica\\
  Università degli Studi dell'Aquila\\
  Via Vetoio (Coppito 1)\\
 I-67100 Coppito\\
Italy}

\email{carlo.scoppola@univaq.it}

\author{Gunnar Traustason}

\address[Gunnar Traustason]
{Department of Mathematical Sciences\\
University of Bath\\
Claverton Down\\
Bath BA2 7AY\\
UK}

\email{gt223@bath.ac.uk}

\urladdr{https://people.bath.ac.uk/gt223/}



\begin{abstract}
  We investigate the occurrence of elements of order $p$ in the upper
  central series of a finite $p$-group.
\end{abstract}

\thanks{The first two authors are members of INdAM---GNSAGA. The first author
gratefully acknowledges support from the Department of Mathematics of
the University of Trento.}

\maketitle

\thispagestyle{empty}

\section{Introduction}

In his Mathematics Stack Exchange post~\cite{Sikora-MSE}, Igor Sikora
asked the following question, which originates with Cihan Okay.
\begin{question}
  \label{question}
  Let $p > 2$ be a prime. Is there a finite $p$-group $G$ with the
  following properties?
  \begin{enumerate}
  \item\label{item:1}
    $G$ is generated by elements of order $p$,
  \item
    $G$ is non-abelian, and
  \item\label{item:3}
    for every pair of non-commuting elements $x, y \in G$ of order
    $p$, their product $x y$ has order 
    greater than $p$.
  \end{enumerate}
\end{question}
A dihedral group of order $2^{n} \ge 8$ provides an example of a group
which satisfies~\eqref{item:1}--\eqref{item:3} for $p  = 2$, hence the
requirement for the prime $p$ to be odd.

In an answer~\cite{aec-ans} to the above post, the first author showed
that there is no group satisfying the properties of 
Question~\ref{question}, building on the following
\begin{lemma}
  \label{lemma:2}
  Let $p$ be an odd prime. Let $G$ be a finite, non-abelian $p$-group,
  which is generated by elements of order $p$.

  Then there is an element of order $p$ in $Z_{2}(G) \setminus Z(G)$.
\end{lemma}
A  negative  answer to  Question~\ref{question}  is  then obtained  as
follows. Let $t \in Z_{2}(G) \setminus  Z(G)$ have order $p$. Since $t
\notin Z(G)$, and $G$ is generated by elements of order $p$, there  is
an element $x  \in G$ of order  $p$ which does 
not commute with $t$. Since $t \in Z_{2}(G)$, we have that $1 \ne [t, x] \in
Z(G)$, so that the group $\Span{x, t}$ has class two, and thus
\begin{equation*}
  (x t)^{p}
  =
  x^{p} t^{p} [t, x]^{\binom{p}{2}}
  =
  [t^{\binom{p}{2}}, x]
  =
  1,
\end{equation*}
as $p > 2$.
Igor   Sikora   kindly   asked   to  include   the   answer   in   the
paper~\cite{Sikora-paper}; it appears there as Lemma 4.8 and its proof.

The goal of this short note is  to record an extension
(Theorem~\ref{thm:p-1} below)
of Lemma~\ref{lemma:2}.
\begin{definition}
  Let $G$ be a finite $p$-group of nilpotence class $c$. 
  \begin{enumerate}
  \item
    For $1 \le i \le c$, the \emph{$i$-th layer}
    of the upper central series is the set
    \begin{equation*}
      Z_{i}(G) \setminus Z_{i-1}(G).
    \end{equation*}
  \item The \emph{spectrum} of $G$ is the set
    \begin{equation*}
      \Set{
        1 \le i \le c
        :
        \text{there is an element of order $p$
          in $Z_{i}(G) \setminus Z_{i-1}(G)$}
      }
    \end{equation*}
  \end{enumerate}
\end{definition}

\begin{theorem}
  \label{thm:p-1}
  Let $p$ be a prime.
  \begin{enumerate}
  \item\label{item:p-1}
    Let $G$ be a finite $p$-group.

    Assume there is an element of order $p$ in the $k$-th layer
    \begin{equation*}
      Z_{k}(G) \setminus Z_{k-1}(G)
    \end{equation*}
    of the upper central series, for some $k \ge 2$.
    
    Then the spectrum of $G$ contains the set
    \begin{equation*}
      \Set{1, 2, \dots, \min\Set{k, p - 1}}.
    \end{equation*}
  \item\label{item:viceversa}
    \begin{enumerate}
    \item
      \label{item:viceversa-a}
      Given any $1 \le k \le p-1$ and $c \ge k$, there is a  finite
      $p$-group  $G$ of class  $c$
      whose spectrum is
      \begin{equation*}
        \Set{ 1, \dots, k }.
      \end{equation*}
    \item
      \label{item:viceversa-b}
      Given any $n \ge 1$, and any sequence
      \begin{equation*}
        p \le c_{1} < c_{2} < \dots < c_{n} \le c
      \end{equation*}
      of integers,  there is a  finite $p$-group  $G$ of class  $c$ whose
      spectrum is
      \begin{equation*}
        \Set{ 1, \dots, p-1, c_{1} , c_{2} , \dots , c_{n} }.
      \end{equation*}
    \end{enumerate} 
  \end{enumerate}
\end{theorem}

\begin{remark}\ppar
  \label{remark}
  \begin{enumerate}
  \item
    Part~\eqref{item:viceversa}  of  Theorem~\ref{thm:p-1} shows  that
    part~\eqref{item:p-1}  provides   the  only  restriction   on  the
    occurrence of  elements of order  $p$ in  the layers of  the upper
    central series of a finite $p$-group.
  \item 
    Not every central series will do in part~\eqref{item:p-1} of
    Theorem~\ref{thm:p-1}; this is discussed in
    Subsection~\ref{subsec:NECSWD}. 
  \end{enumerate}
\end{remark}

We are grateful to Cihan Okay and Igor Sikora for sharing on
Mathematics Stack Exchange the nice
Question~\ref{question}, which led to this note.

\section{Proof of Theorem~\ref{thm:p-1}}
\label{sec:proofs}

\subsection{Proof of part~\eqref{item:p-1}}
\label{subsec:proof}

For $i \leq \min(k, p-1)$ the set
\begin{equation*}
  \Set{
    j
    :
    j \ge i
    \text{, and there is $x \in Z_{j}(G) \setminus
      Z_{j-1}(G)$ of order $p$}
  } 
\end{equation*}
is non-empty, as by assumption it contains $k$. Let $m$ be its minimum.

Suppose by  way of contradiction  that $m >  i$.  Let $x  \in Z_{m}(G)
\setminus Z_{m-1}(G)$ have order $p$.  Since $x \in Z_{m}(G)$, we have
that for all  $y \in G$ the element $[x, y]$ lies  in $Z_{m-1}(G)$. Since $x
\notin Z_{m-1}(G)$, there are elements $y \in G$ such that $[x, y] \in
Z_{m-1}(G)        \setminus        Z_{m-2}(G)$        (see        also
Subsection~\ref{subsec:NECSWD}).   Since $Z_{m-2}(G)  \ge Z_{i-1}(G)$,
for such a $y$ we have  also $[x, y] \notin Z_{i-1}(G)$; in particular
$[x, y] \ne 1$.

Let $s$ be the  greatest integer for which there is  an element $y$ in
$\gamma_{s}(G)$, the $s$-th  term of the lower central  series of $G$,
such that $[x, y] \in  Z_{m-1}(G) \setminus Z_{i-1}(G)$. Since   for
such  a  $y$ the element $[x, y]$ lies in $\gamma_{s+1}(G)$,  we have
that $[x, [x,  y]]$ lies in 
$Z_{i-1}(G)$, and thus in $Z_{p-2}(G)$, as $i \le p - 1$.  Thus
\begin{equation*}
  \frac{\Span{x, [x, y]} Z_{p-2}(G)}{Z_{p-2}(G)}
\end{equation*}
is abelian,  so that the group  $H = \langle  x, [x, y] \rangle  \le H
Z_{p-2}(G)$ has  nilpotence class at  most $p-1$, and is  thus regular
(\cite[III.10.1]{Hup}). Clearly  $x^{y} = y^{-1} x  y = x [x,  y]$ is an
element of order $p$ in $H$, so that by regularity (\cite[III.10.5]{Hup})
\begin{equation*}
  [x, y]^{p}
  =
  (x^{-1} x^{y})^{p}
  =
  1.
\end{equation*}
Thus $[x, y] \in Z_{m-1}(G) \setminus Z_{i-1}(G)$ is an element of
order $p$ in $Z_{t}(G) \setminus 
Z_{t-1}(G)$, for 
some $m > t \ge i$, contradicting the definition of $m$. 

\begin{remark}
  In the proof above, an arbitrary central series could be used
  in the place of the lower central series.

\end{remark}

\subsection{Proof of part~\eqref{item:viceversa}}
\label{subsec:examples}

\newcommand{\thegroup}[0]{E}

We begin by recalling the construction of the unique infinite pro-$p$
group of maximal class, as lifted from~\cite[Section~5]{ranks}. For
the theory of $p$-groups of maximal class, see~\cite{Blackburn}, \cite[III.14]{Hup}, and \cite{L-GMK}.

Let $p$ be a prime, and $\Z_{p}$ be the ring of $p$-adic integers.
Let $\omega$ be a 
primitive $p$-th root of unity. $\omega$ has minimal polynomial
\begin{equation*}
  x^{p-1} + x^{p-2} + \dots + x + 1 \in \Z_{p}[x]
\end{equation*}
over $\Z_{p}$, so  that the ring $\Z_{p}[\omega]$, when  regarded as a
$\Z_{p}$-module, is free of rank $p - 1$.

The ring $\Z_{p}[\omega]$  is a discrete valuation  ring, with maximal
ideal $I = ( \omega - 1  )$. Consider the automorphism $\alpha$ of the
group  $\thegroup  = (\Z_{p}[\omega],  +)$  given  by multiplication  by
$\omega$. Clearly $\alpha$ has order $p$ in $\Aut(\thegroup)$.

The infinite pro-$p$-group of maximal class is the semidirect product
\begin{equation*}
  M
  =
  \Span{\alpha} \ltimes \thegroup.
\end{equation*}
For $p = 2$ this is the infinite pro-$2$-dihedral group, in which all
elements outside $\thegroup$ have order $2$. In general, the following
is well known. 
\begin{lemma}
  \label{fact}
  All elements of $M \setminus \thegroup$  have order $p$. In
  particular, $M$ is generated as a pro-$p$-group by elements of order $p$.
\end{lemma}

\begin{proof}
  If $g \in \thegroup$ and $0 < i < p$,  we
  have
  \begin{align*}
    (\alpha^{i} g)^{p}
    &=
    \alpha^{i p} g^{\alpha^{i(p-1)} + \alpha^{i(p-2)} + \dots + \alpha^{i} + 1}
    \\&=
    g^{\alpha^{i(p-1)} + \alpha^{i(p-2)} + \dots + \alpha^i + 1}
    \\&=
    (\omega^{i(p-1)} + \omega^{i(p-2)} + \dots + \omega^i + 1) g
    \\&=
    0,
  \end{align*}
  since $\omega^i$ is a conjugate of $\omega$, for $0 < i < p$, so that
  it has the same minimal polynomial. 
\end{proof}

Denote by $s_{1} \in E$ the multiplicative unit of $\Z_{p}[\omega]$, and let
$s_{n+1} = [s_{n}, \alpha] = (\omega - 1) s_{n}$, for $n \ge 1$. Then
we have that for $k \ge 2$ the $k$-th term $\gamma_{k}(M)$ of the
lower central series of $M$ is the closed subgroups spanned by $\Set{
  s_{i} : i \ge k }$. Moreover, since $(\omega - 1)^{p-1} = p \zeta$,
for some unit $\zeta$ in $\Z_{p}[\omega]$, we have, for $k \ge 1$,
\begin{equation}
  \label{eq:powers}
  s_{k + p- 1}
  =
  [s_{k}, \underbrace{\alpha, \dots, \alpha}_{p-1}]
  =
  p \zeta s_{k},
\end{equation}
where the repeated commutator is left normed.

For $c  \ge 2$ the  group $M_{c} = M  / \gamma_{c+1}(M)$ is  thus a finite
$p$-group of order $p^{c+1}$ and maximal class $c$,
with
\begin{equation*}
Z_{c-1}(M_{c}) = \Span{s_{2}, \dots , s_{c}} \gamma_{c+1}(M) / \gamma_{c+1}(M).
\end{equation*}
Lemma~\ref{fact} shows that  $M_{c} \setminus Z_{c-1}(M_{c})$ contains
elements    of   order    $p$,    and   it    is   immediately    seen
from~\eqref{eq:powers} that the set of  elements of order dividing $p$
in $Z_{c-1}(M_{c})$ is
\begin{equation*}
  Z_{t}(M_{c})
  =
  \Span{
    {s_{c - (t - 1)}},
    \dots ,
    {s_{c}}
  } \gamma_{c+1}(M)/
  \gamma_{c+1}(M),
\end{equation*}
where $t = \min\Set{c-1, p-1}$. 

Consider also the following
split metacyclic groups, for $c \ge 2$:
\begin{equation*}
  D_{c}
  =
  \begin{cases}
    \Span{ x, y : x^{p^{c}}, y^{p^{c}}, [x, y] = x^{p} },
    &
    \text{for $p > 2$;}
    \\[.75ex]
    \Span{ x, y : x^{2^{c}}, y^{2^{c-1}}, [x, y] = x^{2} },
    &
    \text{for $p = 2$.}
  \end{cases}
\end{equation*}
Thus $D_{c}$  is a group  of class  $c$, with $\Omega_{1}(G)  = Z(G)$
of order $p^{2}$,
where $\Omega_{1}(G) = \Span{
  x^{p^{c-1}},  y^{p^{c-1}}  }$ for $p$  odd,  and
$\Omega_{1}(G) = \Span{  x^{2^{c-1}}, y^{2^{c-2}} }$ for
$p = 2$.

In the following, we will make use of a couple of elementary facts.
\begin{lemma}
  \label{lemma:easy}
  Let $G_{1}$ and $G_{2}$ be non-trivial finite $p$-groups with
  nilpotence classes $c_{1}, c_{2}$ and spectra $S_{1},S_{2}$.

  Let $G_{1}\times G_{2}$ be their direct product.

  Then
  \begin{enumerate}
  \item\label{item:class}
    $G_{1}\times G_{2}$ has class $\max\Set{c_{1},c_{2}}$ and
  \item\label{item:spectrum}
    $G_{1}\times G_{2}$ has spectrum $S_{1}\cup S_{2}$.
  \end{enumerate}
\end{lemma}

\begin{proof}
  For the terms of the upper central series we have
  \begin{equation}
    \label{eq:ucsofproduct}
    Z_{i}(G_{1} \times G_{2})
    =
    Z_{i}(G_{1}) \times Z_{i}(G_{2}),
  \end{equation}
  hence~\eqref{item:class}.

  As to~\eqref{item:spectrum},  if $g$ is  an element of order  $p$ in
  $Z_{i}(G_{1}) \setminus  Z_{i-1}(G_{1})$, say,  then $(g, 1)$  is an
  element  of  order  $p$  in  $Z_{i}(G_{1}  \times  G_{2})  \setminus
  Z_{i-1}(G_{1} \times G_{2})$, according to~\eqref{eq:ucsofproduct}.

  Conversely if $(g, h) \in Z_{i}(G_{1}
  \times G_{2}) = Z_{i}(G_{1}) \times Z_{i}(G_{2})$ has order $p$,
  then $g \in Z_{i}(G_{1})$, $h \in Z_{i}(G_{2})$, and  
  $g^{p} =  1 = h^{p}$. If $(g, h) \notin  Z_{i-1}(G_{1} \times G_{2}) =
  Z_{i-1}(G_{1}) \times Z_{i-1}(G_{2})$, then either $g \notin
  Z_{i-1}(G_{1})$, so that $g$ is an element of order $p$ in
  $Z_{i}(G_{1}) \setminus Z_{i-1}(G_{1})$, or $h \notin
  Z_{i-1}(G_{2})$, so that $h$ is an element of order $p$ in
  $Z_{i}(G_{2}) \setminus Z_{i-1}(G_{2})$.
\end{proof}
We need some further preliminary work in order to produce some
indecomposable examples later.  
Given two non-trivial finite $p$-groups $G_{1}, G_{2}$, 
take elements  $z_{1},z_{2}$ of order $p$  in $Z(G_{1}), Z(G_{2})$
and consider the group  $Q=G/\langle z_{1}z_{2}\rangle$. As the images
$\overline{G_{1}}, \overline{G_{2}}$ of  $G_{1}, G_{2}$  in $Q$  are
isomorphic to  $G_{1}, G_{2}$ it follows that  the class of  $Q$ is
the same as that of $G$.
\begin{lemma}
  Let $x_{i}\in G_{i}$. The following are equivalent
  \begin{enumerate}
  \item\label{lemmaitem:1}
    $\overline{x_{1} x_{2}}\in Z_{n}(Q)$, and
  \item\label{lemmaitem:2}
    $x_{1}\in Z_{n}(G_{1})$ and $x_{2}\in Z_{n}(G_{2})$.
  \end{enumerate}
\end{lemma}
\begin{proof}
  It is immediate that~\eqref{lemmaitem:2} implies~\eqref{lemmaitem:1}

  Now assume~\eqref{lemmaitem:1} In particular 
  \begin{equation*}
    [x_{1} x_{2}, g_{1}, \cdots, g_{n}]
    =
    [x_{1}, g_{1},  \cdots, g_{n}]
    \in  \Span{ z_{1} z_{2} } \cap G_{1}
    =
    1
  \end{equation*}
  for all $g_{1},\ldots ,g_{n}\in  G_{1}$ and therefore
  $x_{1}\in   Z_{n}(G_{1})$.   Similarly   we   see   that   $x_{2}\in
  Z_{n}(G_{2})$.  \endproof
\end{proof}
Note  in particular  that $\overline{x_{1}x_{2}}\not\in  Z_{n}(Q)$ iff
either $x_{1}\not\in Z_{n}(G_{1})$  or $x_{2}\not\in Z_{n}(G_{2})$. It
follows  that $\overline{x_{1}x_{2}}\in  Z_{n}(Q)\setminus Z_{n-1}(Q)$
iff   either   $x_{1}\in  Z_{n}(G_{1})\setminus   Z_{n-1}(G_{1})$   or
$x_{2}\in    Z_{n}(G_{2})\setminus    Z_{n-1}(G_{2})$.   Now    assume
furthermore  that $z_{1}z_{2}$  is  not  a $p$th  power  in $G$.  Then
$\overline{(x_{1}x_{2})}^{p}=1$  if and  only if  $(x_{1}x_{2})^{p}\in
\langle z_{1}z_{2}\rangle$ iff  $(x_{1}x_{2})^{p}=1$ iff $x_{1}^{p}=1$
and $x_{2}^{p}=1$. From this we see that
\begin{prop}
  \label{prop:same}
  Suppose $z_{1}z_{2}$ is not a $p$th power. Then $Q$ and $G$ have the
  same nilpotence class and the same spectrum.
\end{prop}
We cannot drop the requirement that  $z_{1} z_{2}$ is not a $p$th power
as the following example shows.
\begin{example}
  Consider the group $H = D_{2} \times \Span{d}$ where $d$ is
  of order  $p^{2}$. Notice that the  spectrum of $H$ is  $\Set{1}$. Now
  let  $K = H/\Span{x^{p} d^{p}}  $. Then  $\overline{x d}$ is  of
  order $p$  in $K$ and $\overline{x d}\in  Z_{2}(K)\setminus Z(K)$ and
  thus the spectrum of $K$ is $\Set{1,2}$.
\end{example}

\subsection{Proof of part~\eqref{item:viceversa-a}}

For $k = 1$ an example is provided by $D_{c}$.

For $k \ge 2$ a
decomposable example is provided by $M_{k} \times D_{c}$, according to
Lemma~\ref{lemma:easy}. 

We now provide two different indecomposable examples.

\subsubsection{First example}
Consider first the
case when $c = k e$ is a multiple of $k$. Let 
\begin{equation*}
A = \Span{a_{1}} \times \dots \times \Span{a_{k}}
\end{equation*}
be a homocyclic group of exponent $p^{e}$, for some $e \ge 1$. The assignment
\begin{equation*}
a_{1} \mapsto a_{1} a_{2},
a_{2} \mapsto a_{2} a_{3},
\dots,
a_{k-1} \mapsto a_{k-1} a_{k},
a_{k} \mapsto a_{k} a_{1}^ {p}
\end{equation*}
defines  an  endomorphism  $\beta$  of   $A$,  which  is  actually  an
automorphism of  order a power of  $p$, as it is  such an automorphism
modulo  the Frattini  subgroup of  $A$. Let  $p^{t}$ be  the order  of
$\beta$.

Consider the  group $G$ which  is the semidirect  product of $A$  by a
cyclic group $\Span{b}$ of order $p^{t+1}$, with $b$ acting on $A$ via
$\beta$, so that $b^{p^{t}} \in Z(G)$. ($G$ coincides with $D_{e}$ for
$k = 1$.)

If $h a  \in G$ is an element  of order dividing $p$, for  some $h \in
\Span{b}$ and  $a \in A$,  then its  projection $h$ on  $\Span{b}$ has
also order dividing  $p$, and thus $h \in  \Span{b^{p^{t}}} \le Z(G)$,
so that $a \in \Omega_{1}(A)$. It  follows that the elements of $G$ of
order  dividing $p$  form the  set $\Span{\Omega_{1}(A),  b^{p^{t}}} =
\Span{a_{1}^{p^{e-1}},  \dots, a_{k}^{p^{e-1}},  b^{p^{t}}}$. Now  the
element $a_{k-i+1}^{p^{e-1}}$ of order $p$ lies in $Z_{i}(G) \setminus
Z_{i-1}(G)$, for $1 \le i \le k$.

For the general case when $c$ is not a  multiple of $k$, if $c = k e -
s$ for  some $e  \ge 1$ and  $1 \le s  < k$,  it suffices to  take the
subgroup of the $G$ we just constructed given by
\begin{equation*}
  \Span{a_{1}^{p}, \dots, a_{s}^{p}, a_{s+1}, \dots, a_{k}, b}.
\end{equation*}

\subsubsection{Second example}

Let $2\leq k\leq p-1$ and $c\geq k$. As an application of
Proposition~\ref{prop:same}, we produce an example of an
indecomposable finite $p$-group $G_{(k,c)}$ that has class $c$ and
spectrum $\Set{1,\cdots ,k}$.

Let $G_{1}  = D_{c}$  and $G_{2} = \langle s, t \rangle$ be  the largest
$2$-generator group of exponent $p$ and class $k$. Then $G=G_{1}\times
G_{2}$ has  the class and  spectrum we want. We  want a group  that is
furthermore indecomposable. Pick $d\in \gamma_{k}(G_{2})$ and consider
the group $Q=G/\Span{ x^{p^{c-1}} d }$.  By Proposition~\ref{prop:same} we know
that $Q$ has  the same class and  spectrum as $G$. It  remains to show
that $Q$ is indecomposable. We argue by contraction and suppose that
\begin{equation*}
  Q = A \times B,
\end{equation*}
where $A$  and $B$ are  non-trivial. This is a  group of rank  $4$ and
among the generators of $A$ and  $B$ there must be two generators, one
in 
\begin{equation*}
  \overline{x^{r} y^{s}G_{2}[G_{1},G_{1}]G_{1}^{p}}
\end{equation*}
and the other in
\begin{equation*}
  \overline{x^{u} y^{v}G_{2}[G_{1},G_{1}]G_{1}^{p}}
\end{equation*}
where $x^{r} y^{s}$
and $x^{u} y^{v}$ are linearly independent modulo the Frattini subgroup
$[G_{1},G_{1}]G_{1}^{p}$.               As               $[x, y]\not\in
[G_{1},G_{1}]^{p}\gamma_{3}(G_{1})$, these  two   generators  must
belong  to the  same component, say  $A$.
Replacing  these with  some
appropriate products of their powers
of  the two,  we can
assume that these are in $\overline{x G_{2}[G_{1},G_{1}]G_{1}^{p}}$ and
$\overline{y G_{2}[G_{1},G_{1}]G_{1}^{p}}$.   Again   as  $[x, y]\not\in
[G_{1},G_{1}]^{p}\gamma_{3}(G_{1})$ all the generators  in $B$ must be
in $\overline{G_{2}[G_{1},G_{1}]G_{1}^{p}}$.

A similar argument
shows  that  there  must  be  two   generator  for  $B$  that  are  in
$\overline{s [G,G]G^{p}}$  and   in  $\overline{t [G,G]G^{p}}$   and  as
$[s, t]\not\in [G_{2},G_{2}]G_{2}^{p}$ we furthermore  see that the two
generators  for $A$  are  in  $\overline{x [G,G]G^{p}}$ and  $\overline
{y [G,G]G^{p}}$.

Using the fact that $A$ and $B$ are normal in $Q$. We then see that 
\begin{equation*}
  A
  \geq
      [A, \underbrace{\overline{G_{1}}, \dots, \overline{G_{1}}}_{k-1}]
      =
      \gamma_{k}(\overline{G_{1}})
\end{equation*}
and
\begin{equation*}
  B
  \geq
      [B,\underbrace{\overline{G_{2}}, \dots, \overline{G_{2}}}_{c-1}]
      =
      \gamma_{c}(\overline{G_{2}}).
\end{equation*}
But then $A$ contains $\overline{x^{p^{c-1}}}$ and $B$ contains
$\overline{d^{-1}}$. As these elements are the same we get the
contradiction that $A\cap B \neq 1$.  

\subsection{Proof of part~\eqref{item:viceversa-b}} 

Note that the assumptions imply $c \ge p$. In the case when $n = 1$ and
$c_{1} = c = p$, an indecomposable example is given by $M_{p}$. So from
now on we may assume $c > p \ge 2$.

\subsubsection{A decomposable example}

By Lemma~\ref{lemma:easy}, such an example is given by
\begin{equation}
  \label{eq:dec}
  H
  =
  M_{c_{1}} \times M_{c_{2}} \times \dots \times M_{c_{n}} \times D_{c},
\end{equation}
where the factor $D_{c}$ is redundant when $c = c_{n}$.

This  example also shows that the
layers of the upper 
central series that contain elements of order $p$ may well contain
also elements of order greater than $p$.

\subsubsection{An indecomposable example}

We will construct an indecomposable example as a subgroup $G$ of the
decomposable one $H$ of~\eqref{eq:dec}.

Denote by $a_{i}, t_{i, 1}, \dots, t_{i, c_{i}}$ the images in $M_{c_{i}}$ of
$\alpha, s_{1}, \dots, s_{c_{i}}$.

For each $i$, consider the element $x_{i} = a_{i} t_{1, i}$ of order
$p$ of $M_{c_{i}}$, which together
with $a_{i}$ generates $M_{c_{i}}$. Consider for each $i$ the
non-abelian maximal subgroup 
\begin{equation*}
  X_{i}
  =
  \Span{
    x_{i}, \gamma_{2}(M_{c_{i}})
    }
\end{equation*}
of $M_{i}$, and the maximal subgroup
\begin{equation*}
  E
  =
  \Span{x, y^{p}}
\end{equation*}
of $D_{c}$. Write
\begin{equation*}
  \extra{X_{i}}
  =
  \Set{1} \times \dots \times
  \underbrace{X_{i}}_{i-\text{th place}}
  \times \dots \Set{1}
\end{equation*}
and
\begin{equation*}
  \extra{E}
  =
  \Set{1} \times \dots \times \Set{1} \times E.
\end{equation*}
For $z \in M_{c_{i}}$, write
\begin{equation*}
  \extra{z}
  =
  (1, \dots, 1,
  \underbrace{z}_{i-\text{th place}},
  1, \dots , 1) \in \extra{X_{i}},
\end{equation*}
and for $z \in E$ write
\begin{equation*}
  \extra{z}
  =
  (1, \dots, 1, z) \in \extra{E}.
\end{equation*}
Let also
\begin{equation*}
  a
  =
  (a_{1}, \dots, a_{n}, y)
  \in H.
\end{equation*}

Our example will be the subgroup of $H$ given by
\begin{equation*}
  G
  =
  \Span{
    a,
    \extra{X}_{1}, \dots, \extra{X}_{n}, \extra{E}
  }.
\end{equation*}

For $z \in X_{i}$  one has
\begin{equation*}
  [\extra{z}, a]
  =
  \extra{[z, a_{i}]} \in \extra{X_{i}},
\end{equation*}
and for $z \in E$ one has
\begin{equation*}
  [\extra{z}, a]
  =
  \extra{[z, y]} \in \extra{E}.
\end{equation*}
It follows that
\begin{equation*}
  \Span{a, \extra{X_{i}}}
  \cong
  M_{c_{i}},
  \quad
  \Span{a, \extra{E}}
  \cong
  D_{c}.
\end{equation*}
and for each $k \ge 2$ we have the equality
\begin{equation}
  \label{eq:gammas}
  \gamma_{k}(M_{c_{1}}) \times \dots \times \gamma_{k}(M_{c_{n}})
  \times
  \gamma_{k}(D_{c})
  =
  \gamma_{k}(G).
\end{equation}

The centre of $H$ is
\begin{equation*}
  Z(H)
  =
  \Span{
    \extra{t_{1, c_{1}}}, \dots, \extra{t_{n, c_{n}}},
    \extra{x^{p^{c-1}}}, \extra{y^{p^{c-1}}}. 
    }
\end{equation*}
for $p$ odd, and
\begin{equation*}
  Z(H)
  =
  \Span{
    \extra{t_{1, c_{1}}}, \dots, \extra{t_{n, c_{n}}},
    \extra{x^{2^{c-1}}}, \extra{y^{2^{c-2}}. }
    }
\end{equation*}
for $p = 2$.
We have $Z(H) \le G$, and
\begin{equation*}
  Z(H)
  =
  C_{H}(a) \cap C_{H}(\extra{x}),
\end{equation*}
so that  $Z(G) = Z(H)$.

Moreover  $Z(G)$ is contained in  the Frattini
subgroup $\Frat(G)$ of $G$, as
\begin{equation*}
  \extra{t_{i, c_{i}}} \in \extra{\gamma_{c_{i}}(M_{c_{i}})},
  \quad
  \extra{x^{p^{c-1}}} \in \Span{\extra{x^{p}}}
\end{equation*}
and
\begin{equation}
  \label{eq:itsapthpower}
  a^{p} = \extra{y^{p}}
\end{equation}
(recall $c \ge 3$).
Therefore $G$ has  no non-trivial
abelian direct factor.

Now $H / Z(H)$ and $G / Z(G)$ have the same structure as $H$ and $G$,
for parameters 
\begin{equation*}
  c_{1} - 1, \dots, c_{n} - 1, c - 1
\end{equation*}
(dropping $c_{1} - 1$ if it equals $p - 1$). Therefore $Z_{c}(G) = G$,
and for all $j < c$ the 
$j$-th of the 
upper central series of $G$ coincides with
\begin{equation*}
  (X_{1} \cap Z_{j}(M_{c_{1}}))
  \times \dots \times
  (X_{n} \cap Z_{j}(M_{c_{n}}))
  \times
  (E \cap Z_{j}(D_{c})),
\end{equation*}
so that the class and the spectrum of $G$ are as required.

Suppose now, by way of contradiction, that $G$ admits a non-trivial
decomposition $G = G_{1} \times 
G_{2}$. Since $U = \Span{ \extra{X_{1}}, \dots, \extra{X_{n}},
  \extra{E} }$ is a maximal 
subgroup of $G$, there will be an element of $G_{1}$, say, of the form
$a u$, for some $u \in U$. Consider a fixed set $L$ of minimal
generators of $G_{2}$; these will be part of a set of minimal
generators of $G$.

Assume there is an element $v \in L$ which is contained in $U$. Then
\begin{equation*}
  1
  =
  [a u, v]
  =
  [a, v]^{u} [u, v].
\end{equation*}
Since $[u, v] \in \gamma_{2}(U) \le \gamma_{3}(G)$, we obtain
$[a, v] \in \gamma_{3}(G)$. Write
\begin{equation*}
  v = (v_{1}, \dots, v_{n}, d),
\end{equation*}
so that
\begin{equation*}
  [a, v]
  =
  ([a_{1}, v_{1}], \dots, [a_{n}, v_{n}],
  [y, d]).
\end{equation*}
Since the groups $M_{c_{i}}$ are of maximal class, and $[a_{i}, v_{i}]
\in \gamma_{3}(M_{c_{i}})$, we obtain $v_{i}
\in \gamma_{2}(M_{c_{i}})$. Also, $[y, d] \in \gamma_{3}(D_{c}) = \Span{x^{p^{2}}}$
implies $d \in \Span{x^{p}, y^{p}}$.

From~\eqref{eq:gammas}~and \eqref{eq:itsapthpower} it  follows that $v
\in \Frat(G)$, contradicting the fact that  $v$ is in a minimal set of
generators of $G$.

Therefore all elements of $L$  can be assumed, taking suitable powers,
to be  of the form $a  w$, for some $w  \in U$.  Since the  product of
such an element by the inverse of  another of the same form is in $U$,
we obtain that $L$ has only one element, so that $G_{2}$
is a non-trivial
cyclic direct factor of $G$, a contradiction.

\subsection{Not every central series will do}
\label{subsec:NECSWD}

In a group $G = D_{c}$, for $p > 2$ and $c \ge 2$, we have
$\gamma_{2}(G) =  \Span{x^{p}}$, and $y^{p^{c-1}}$
is an element of order $p$ in the first layer
$G \setminus \gamma_{2}(G)$
of the lower central series.  But then the only
other layer of the lower 
central series of $G$
that contains elements of order $p$ is the last non-empty one
$\gamma_{c}(G) \setminus \Set{1} = \Span{x^{p^{c-1}}} \setminus \Set{1}$.

In fact,  in the proof  of Subsection~\ref{subsec:proof} we  have used
the  implication $\eqref{item:ucs} \implies \eqref{item:quote}$  from
the  following  characterisation  of the  upper 
central series, which is presumably well-known.
\begin{lemma}
  Let $G$ be a group, let $c \ge 1$, and let
  \begin{equation}
    \label{eq:a_series}
    \Set{1} = G_{0} < G_{1} < \dots < G_{c} = G
  \end{equation}
  be a central series.  The following are equivalent:
  \begin{enumerate}
  \item 
    \label{item:quote}
    for every $2 \le m \le c$ and every $x \in G_{m} \setminus
    G_{m-1}$, there is $y \in G$ such that $[x, y] \in G_{m-1}
    \setminus G_{m-2}$;
  \item
    \label{item:ucs}
    the series~\eqref{eq:a_series} coincides with the upper central
    series, that is, $G_{i} = Z_{i}(G)$ for all $1 \le i \le c$.  
  \end{enumerate}
\end{lemma}

\begin{proof}
  The upper central series $G_{i} = Z_{i}(G)$  clearly
  satisfies~\eqref{item:quote}. In fact, if $x \in Z_{m}(G)$, for some
  $m \ge 2$, is such that $[x, y] \in Z_{m-2}(G)$ for all $y \in G$,
  then $x \in Z_{m-1}(G)$, as $Z_{m-1}(G) / Z_{m-2}(G) = Z(G /  Z_{m-2}(G))$.

  Conversely,       assume        the       series~\eqref{eq:a_series}
  satisfies~\eqref{item:quote}, and  proceed by induction on  $c$. The
  statement  is vacuously  true for  $c  = 1$,  that is,  when $G$  is
  non-trivial and abelian. Let $c \ge 2$. We have $G_{1} \le Z(G)$. If
  by way of  contradiction $G_{1} < Z(G)$, pick $z  \in Z(G) \setminus
  G_{1}$. There will be an $m \ge  2$ such that $z \in G_{m} \setminus
  G_{m-1}$  but  $[z,  y]  =  1  \in  G_{m-2}$  for  all  $y  \in  G$,
  defeating~\eqref{item:quote}. Thus $G_{1} = Z(G) = Z_{1}(G)$.

  Now if $x G_{1} \in (G_{m} / G_{1}) \setminus (G_{m-1} / G_{1})$, for
  some $m \ge 3$, we have $x \in G_{m} \setminus
  G_{m-1}$. By~\eqref{item:quote} there is $y \in G$ for which $[x,
    y] \in G_{m-1} 
  \setminus G_{m-2}$, so that $[x G_{1}, y G_{1}] \in (G_{m-1} / G_{1})
  \setminus (G_{m-2} / G_{1})$.
  We can thus apply the inductive hypothesis to $G/G_{1}$ to
  obtain~\eqref{item:ucs}. 
\end{proof}

%% \bibliographystyle{amsalpha}
 
%% \bibliography{Refs}
\providecommand{\bysame}{\leavevmode\hbox to3em{\hrulefill}\thinspace}
\providecommand{\MR}{\relax\ifhmode\unskip\space\fi MR }
\providecommand{\MRhref}[2]{%
  \href{http://www.ams.org/mathscinet-getitem?mr=#1}{#2}
}
\providecommand{\href}[2]{#2}

\end{document}